\documentclass[a4paper,10pt]{article}
\usepackage[utf8]{inputenc}
\usepackage{amsmath,amsthm,amssymb,amsfonts,graphicx,bm}
\usepackage[pdftex]{hyperref}
\usepackage[ruled]{algorithm2e}		%	algorithm
\usepackage{multirow}
\usepackage{caption}
\usepackage{pgfplots} 
\pgfplotsset{compat=newest}

\usetikzlibrary{
    patterns,
}
    \pgfplotsset{
        compat=1.7,
        my ybar legend/.style={
            legend image code/.code={
                \draw [##1] (0cm,-0.6ex) rectangle +(2em,1.5ex);
            },
        },
    }

\def\keywordname{{\bfseries Keywords}}%
\def\keywords#1{\par\addvspace\medskipamount{\rightskip=0pt plus1cm
\def\and{\ifhmode\unskip\nobreak\fi\ $\cdot$
}\noindent\keywordname\enspace\ignorespaces#1\par}}
\def\subclassname{{\bfseries Mathematics Subject Classification
(2000)}\enspace}
\def\subclass#1{\par\addvspace\medskipamount{\rightskip=0pt plus1cm
\def\and{\ifhmode\unskip\nobreak\fi\ $\cdot$
}\noindent\subclassname\ignorespaces#1\par}}

\theoremstyle{plain}
\newtheorem{thm}{Theorem}[section]

\newtheorem{obs}[thm]{Observation}
\newtheorem{rmk}[thm]{Remark}

\newcommand{\RR}{{\mathbb{R}}}
\newcommand{\CC}{{\mathbb{C}}}

\newcommand{\AAA}{\bm{A}}
\newcommand{\BBB}{\bm{B}}
\newcommand{\CCC}{\bm{C}}
\newcommand{\SSS}{\bm{S}}
\newcommand{\OOO}{\mathcal{O}}
\newcommand{\xxx}{\bm{x}}
\newcommand{\qqq}{\bm{q}}
\newcommand{\eee}{\bm{e}}

\DeclareMathOperator{\vecc}{vec}
\DeclareMathOperator{\Err}{Err}
\DeclareMathOperator{\rk}{rank}
\DeclareMathOperator{\diag}{diag}
\DeclareMathOperator{\sspan}{span}

\title{The infinite Lanczos method for symmetric nonlinear eigenvalue problems}

\author{Giampaolo Mele
\thanks{Department of Mathematics, KTH Royal Institute of Technology, SeRC swedish e-science research center, Lindstedtsv\"agen 25, SE-100 44 Stockholm, Sweden, email: gmele,@kth.se}
}

\date{\today}

\begin{document}

\maketitle

\begin{abstract}
A new iterative method for solving large scale symmetric nonlinear eigenvalue problems is presented. We firstly derive an infinite dimensional symmetric linearization of the nonlinear eigenvalue problem, then we apply the indefinite Lanczos method to this specific linearization, resulting in a short-term recurrence. We show how, under specific assumption on the starting vector, this method can be carried out in finite arithmetic and how the exploitation of the problem structure leads to improvements in terms of computation time. The eigenpair approximations are extracted with the nonlinear Rayleigh–Ritz procedure combined with a specific choice of the projection space. We illustrate how this extraction technique resolves the instability issues that may occur due to the loss of orthogonality in many standard Lanczos-type methods.

%Include keywords and mathematical subject classification numbers as needed.
\subclass{35P30 \and 65H17 \and 65F60 \and 15A18 \and 65F15}
\end{abstract}

\section{Introduction}

% General introduction on nonlinear eigenvalue problems
We consider the nonlinear eigenvalue problem (NEP) which consists of computing $(\lambda, v) \in D \times \CC^n \setminus \left \{0 \right \}$ such that
\begin{align} \label{eq:nep}
M(\lambda) x = 0,
\end{align}
where
\begin{align} \label{eq:spmf}
 M(\lambda) = \sum_{m=1}^p f_m(\lambda) A_m,
\end{align}
with $D \subset \CC$ open disk, $f_m: D \rightarrow \CC$ analytic functions, and $A_m \in \CC^{n \times n}$ for $m=1, \dots, p$. In this work we focus on the symmetric NEP, namely we assume that 
\begin{align} \label{eq:Am_symm}
&& M(\lambda)^T=M(\lambda)	&& \forall \lambda \in D.
\end{align}
Equivalently, $M(\lambda)$ can be expressed as~\eqref{eq:spmf} with symmetric matrices $A_m^T = A_m$ for $m=1, \dots, p$. Notice that the complex matrices $A_m$ and $M(\lambda)$ are assumed to be symmetric but not necessarily Hermitian. The NEP arises in many areas such as: stability analysis, control theory, wave propagation, etc, and it has been studied in various settings. See the review papers \cite{guttel2017nonlinear,mehrmann2004nonlinear}, the PhD theses \cite{effenberger2013robust,van2015rational}, and the problem collection \cite{betcke2013nlevp}. Specialized software for NEPs has been recently produced: the package NEP-PACK~\cite{NEP-PACK}, the library SLEPc~\cite{Hernandez:2005:SSF}, and even more open-source software. 
% Overview about linearizations and structure preserving linearizations
An approach for solving the NEP consists of constructing a linear eigenvalue problem (linearization) whose eigenvalues approximate, or correspond to, the eigenvalues of the original NEP~\cite{amiraslani2008linearization,lawrence2016backward,antoniou2004new,mackey2006vector,gohberg2005matrix,de2010fiedler,robol2017framework}. When the NEP has specific structures such as being: symmetric, Hermitian, Hamiltonian, palindromic, etc, it is preferable to construct a linearization that preserves these structures. Theoretical and algorithmic aspects of structured linearizations have been extensively analyzed \cite{mackey2006palindromic,su2011solving,mehrmann2002polynomial,mackey2009numerical,bueno2014structured,bueno2014structured2,fassbender2018block}. In particular, it has been shown that methods based on structure preserving linearizations, in certain applications, are more robust than other methods that do not take into account the structure~\cite{mehrmann2001structure,mackey2006structured}. For the polynomial eigenvalue problem (PEP), i.e., the special case where $f_m(\lambda)$ in~\eqref{eq:nep} are polynomials, 
symmetric linearizations are extensively characterized in~\cite{higham2006symmetric,bueno2018block}. A well established class of methods for solving symmetric eigenvalue problems are Lanczos-like methods.
% Short history of Lanczos method
More precisely, the Lanczos method, and its variants, can be applied for solving symmetric generalized eigenvalue problems $A x = \lambda B x$ where $A,B \in\CC^{n \times n}$ are symmetric matrices. The original approach~\cite{Lanczos:1950zz} was developed for the case $B=I$, a generalization for $B$ positive definite is presented in \cite[Ch.15, Sec.11]{parlett1998symmetric}. A further extension of this approach, known as indefinite Lanczos method, for the case where $A$ and $B$ are Hermitian or symmetric matrices is discussed in~\cite{parlett1990use} and~\cite[Section 8.6.1]{bai2000templates}. Lanczos methods belong to the class of Krylov methods. They exploit the fact that, during the construction of the orthogonal basis of the Krylov space, due to the symmetry of the problem, the orthogonalization can be performed in a more efficient way with a three-term recurrence. The main disadvantage of having a three-term recurrence is that, in floating-point arithmetic and without further specializations, the basis vectors are often affected by loss of orthogonality, resulting in slow convergence of the Ritz pairs and numerical instability~\cite{parlett1979lanczos,simon1984lanczos,wu2000thick,abdel2010deflated}.

% short and strong statements about the paper results
In this work, we present a new symmetric linearization for the symmetric NEP, resulting in a symmetric, linear, and infinite dimensional eigenvalue problem. Symmetric generalized eigenvalue problems can be solved with the indefinite Lanczos method~\cite[Section 8.6.1]{bai2000templates}. We present a new method that corresponds to adapting, in an efficient and robust way, the indefinite Lanczos method to the derived linearization. In order to cure the slow convergence, that is due to the loss of orthogonality, we use the nonlinear Rayleigh–Ritz procedure, combined with a proper choice of the projection space, for extracting eigenpair approximations by exploiting the structure of the derived linearization. The proposed method is competitive, in terms of robustness and complexity, with Arnoldi-like  methods for NEPs that perform the full orthogonalization.

% organization of the paper
The paper is organized as follows: in Section~\ref{sec:derivation} we prove that the symmetric NEP is equivalent to a 
symmetric, linear, and infinite dimensional eigenvalue problem. In Section~\ref{sec:Lanc_finite} we derive a method, in finite arithmetic, which consists of applying the indefinite Lanczos method to the derived linearization. In Section~\ref{sec:Z_comp} we show how the computation time of the resulting method can be considerably reduced by exploiting additional NEP structures. In Section~\ref{sec:num_exp} we illustrate the performance of this new approach with numerical simulations by solving large and sparse NEPs. %We compare performances with other well established methods. 
The simulations were carried out in the Julia programming language~\cite{bezanson2017julia} with NEP-PACK~\cite{NEP-PACK}, which is an open source Julia package for NEPs.

% similar methods
The method we derive can been seen as an Arnoldi-like method applied to an iteratively expanding linearization, or equivalently to a infinite dimensional linear operator. Other methods that are based on these ideas  are: infinite Arnoldi~\cite{Jarlebring_INFARN_2012} and its tensor variant~\cite{jarlebring2017waveguide}, NLEIGS~\cite{guttel2014nleigs}, and CORK~\cite{van2015compact}. There are also methods based on the bi-orthogonalization procedure, which also lead to a three-term recurrence, presented in \cite{lietaert2018compact,gaaf2017infinite}. However, these methods and their variations, in the way they are presented and without further research, are not capable of taking advantage of the symmetry of the NEP. 

In the rest of this work, vectors and matrices are denoted as $v = [v_i]_{i=1}^{n}$ and $A=[a_{i,j}]_{i,j=1}^{n,m}$ respectively, whereas bold letters represent block-vectors and block-matrices with infinite size, namely $\bm v = [v_i]_{i=1}^{\infty}$ with $v_i \in \CC^n$ and $\bm A=[A_{i,j}]_{i,j=1}^{\infty}$ with $A_{i,j} \in \CC^{n \times n}$. The matrix $[\bm A]_k = [A_{i,j}]_{i,j=1}^{k} \in \CC^{n k \times n k}$ consists of the main sub-matrix obtained by extracting the first $k$-blocks. The Kronecker product and the Hadamard (element-wise) product are denoted by $\otimes$ and $\circ$ respectively. The vectors $e_j$ and $\bm e_j$ have zeros as elements except one in the $j$-th position whereas $e$ and $\bm e$ are the vectors with all ones. Without loss of generality, after a change of variables in  \eqref{eq:spmf}, we assume that the region of interest $D \subset \CC$ is a disk centered in the origin. The derivatives of \eqref{eq:spmf} will be denoted as $M_i:=M^{(i)}(0)$. We will denote by $A^T$ the transpose (not conjugate transpose) of the matrix $A \in \CC^{n \times n}$.

\section{Indefinite Lanczos method in infinite dimensional settings} \label{sec:derivation}
In order to derive a symmetric linearization for NEPs, we first review a specific linearization technique for PEPs. This technique consists of symmetrizing the companion linearization and it is presented in \cite[Theorem 2.1]{mehrmann2002polynomial} that we recall below. The approach is previously reported in \cite[Ch.4 Sec. 2]{lancaster1966lambda} for the scalar case.

\begin{thm}[Mehrmann and Watkins \cite{mehrmann2002polynomial}, Lancaster \cite{lancaster1966lambda}]\label{thm:PEPsym}
Consider the polynomial eigenvalue problem $M(\lambda)v=0$ where $M(\lambda)=\sum_{j=0}^k M_j \lambda^j$ and $M_k$ nonsingular. Then, the pencil $A-\lambda B \in \CC^{nk \times nk}$, where 
\setlength\arraycolsep{3.3pt}
\begin{align} \label{eq:PEP_sym_pencil}
 A = 
 \begin{pmatrix}
  -M_0 	&	0	&	0	&	0	& \dots	&	0	\\
  0	&	M_2	&	M_3	&	M_4	& \dots &	M_k	\\
  0	&	M_3	&	M_4	&		&	&	0	\\
  0	&	M_4	&		&		&	&	0	\\
 \vdots	&	\vdots	&		&		&	&	\vdots  \\
  0	&	M_k	&	0	&	0	& \dots	&	0	\\
 \end{pmatrix}, &&
  B = 
 \begin{pmatrix}
  M_1 	&	M_2	&	M_3	&	\dots	& M_{k-1}	&	M_k	\\
  M_2	&	M_3	&	M_4	&		& M_k	 	&	0	\\
  M_3	&	M_4	&		&		&		&	0	\\
  M_4	&		&		&		&		&	0	\\
 \vdots	&		&		&		&		&	\vdots  \\
  M_k	&	0	&	0	&	0	& \dots		&	0	\\
 \end{pmatrix}
\end{align}
has the same eigenvalues of $M(\lambda)$. If $M_j$ are symmetric, i.e., $M(\lambda)$ is symmetric, then $A$ and $B$ are symmetric. If $M(\lambda) v = 0 $, then 
$[v^T, \lambda v^T, \dots, \lambda^{k-1} v^T]^T$ is an eigenvector of $A-\lambda B$.
\end{thm}

The proof is based on the following argument. A pair $(\lambda,x)$ that fulfill $M(\lambda)x=0$ defines an eigenpair of the companion linearization~\cite{gohberg2005matrix}, namely
\setlength\arraycolsep{3.3pt}
\begin{align}\label{eq:PEP_companion}
 \begin{pmatrix}
  -M_0 	&		&		& 		&		\\
  	&	I	&		& 		&		\\
  	&		&	I	&		&		\\
 	&		&		& \ddots 	&		\\
  	&		&		& 		&	I	\\
 \end{pmatrix}
\begin{pmatrix}
 x 		\\
 \lambda x 	\\
 \lambda^2 x 	\\
 \vdots 		\\
 \lambda^{k-1} x
\end{pmatrix}
=  \lambda
 \begin{pmatrix}
  M_1 	&	M_2	&	\dots	&	M_{k-1}	&	M_k	\\
  I	&		&		&		&	0	\\
  	&	I	&		&		&	0	\\
 	&		&	 \ddots	&		&	\vdots  \\
 	&		&		&	I	&	0	\\
 \end{pmatrix}
 \begin{pmatrix}
 x 		\\
 \lambda x 	\\
 \lambda^2 x 	\\
 \vdots 	\\
 \lambda^{k-1} x
\end{pmatrix}.
\end{align}
To obtain~\eqref{eq:PEP_sym_pencil} we multiply~\eqref{eq:PEP_companion} on the left by the matrix 
\begin{align} \label{eq:PEP_symmetrizer}
S =
  \begin{pmatrix}
  I 	&	0	&	0	&	0	& \dots	&	0	\\
  0	&	M_2	&	M_3	&	M_4	& \dots &	M_k	\\
  0	&	M_3	&	M_4	&		&	&	0	\\
  0	&	M_4	&		&		&	&	0	\\
 \vdots	&	\vdots	&		&		&	&	\vdots  \\
  0	&	M_k	&	0	&	0	& \dots	&	0	\\
 \end{pmatrix}.
\end{align}

The main disadvantage of using this linearization in practice is that the blocks forming the eigenvectors of the pencil, defined by~\eqref{eq:PEP_sym_pencil}, grow or decay exponentially, depending on the value of $|\lambda|$. More precisely, the norm of the $j$-th block is $|\lambda|^j \| x \|$ and, if the PEP has high degree, this leads to overflow or underflow when the eigenpairs of~\eqref{eq:PEP_sym_pencil} are computed numerically. In order to resolve this issue, in this section we consider the scaled companion linearization presented \cite[Section 5.1]{Jarlebring_INFARN_2012}. We extend the ideas used in Theorem~\ref{thm:PEPsym} to symmetrize the scaled companion linearization. Moreover, we consider the NEP in its general form \eqref{eq:nep}, therefore we derive a linearization that involves vectors and matrices with infinite size.

\subsection{An infinite dimensional symmetric linearization}
The NEP \eqref{eq:nep} is equivalent to a linear and infinite dimensional eigenvalue problem, see \cite[Section 5.1]{Jarlebring_INFARN_2012}. More precisely, if $(\lambda,x)$ is an eigenpair of \eqref{eq:nep}, the following relation between vectors and matrices of infinite size is fulfilled
\renewcommand{\arraystretch}{1.3}
\begin{align} \label{eq:linNEP}
 \begin{pmatrix}
  -M_0	&		&		& 		&	       \\
        &	I	&		& 		&	       \\
        &		&	I	& 		&	       \\
        &		&		& 	I	&	       \\
        &		&		& 		& \ddots		
\end{pmatrix}
\begin{pmatrix}
 \frac{\lambda^0}{0!} x  \\
 \frac{\lambda^1}{1!} x  \\
 \frac{\lambda^2}{2!} x  \\
 \frac{\lambda^3}{3!} x  \\
 \vdots	 	
\end{pmatrix}
= \lambda
 \begin{pmatrix}
  M_1		&	\frac{1}{2} M_{2}		&	\frac{1}{3} M_{3}		& 	\frac{1}{4} M_4	& \dots		\\
  \frac{1}{1} I	&					&					& 			&		\\
		&	\frac{1}{2}I			&					& 			&		\\
		&					&	\frac{1}{3}I			&			&		\\
		&					&					&	\ddots		&		\\
\end{pmatrix}
\begin{pmatrix}
 \frac{\lambda^0}{0!} x  \\
 \frac{\lambda^1}{1!} x  \\
 \frac{\lambda^2}{2!} x  \\
 \frac{\lambda^3}{3!} x  \\ 
 \vdots	 	
\end{pmatrix}. 
\end{align}
\renewcommand{\arraystretch}{1}
The equation~\eqref{eq:linNEP} defines a linear and infinite dimensional eigenvalue problem 
\begin{align} \label{eq:nep_lin}
\AAA \xxx = \lambda \BBB \xxx,
\end{align}
where $\AAA, \BBB, \xxx$ are matrices and vector of infinite size defined accordingly. Clearly, the linearization~\eqref{eq:linNEP} is never symmetric. However, if the NEP is symmetric, i.e., it holds~\eqref{eq:Am_symm}, then it is possible to symmetrize~\eqref{eq:linNEP} with a similar technique as in Theorem~\ref{thm:PEPsym}. More precisely, since we consider a scaled and infinite companion linearization, in the following theorem we derive a scaled and infinite version of the matrix~\eqref{eq:PEP_symmetrizer} that symmetrizes~\eqref{eq:linNEP}. 

\begin{thm}[Symmetric linearization]
Assume that the NEP~\eqref{eq:spmf} is symmetric, i.e., it holds~\eqref{eq:Am_symm}, then there exists a unique matrix $\CCC$ such that
 \begin{align} \label{eq:symmetrizer}
 \SSS:= 
 \left[
 \begin{pmatrix}
	  1	&		\\
		&	\CCC	
 \end{pmatrix}
 \otimes
 \eee \eee^T 
 \right]
 \circ
  \begin{pmatrix}
   I	&			&			&		   	&       &       \\
	&	M_{2}		&	M_{3}		&	M_{4}	   	& M_{5} & \dots	\\
	&	M_{3}		&	M_{4}		&	M_{5}	   	&       &	\\
	&	M_{4}		&	M_{5}		&		   	&       &	\\
	&	M_{5} 	    	&			&		   	&       &       \\
	&   \vdots      	&               	&			&       &
  \end{pmatrix}
 \end{align}
 is a symmetrizer for \eqref{eq:linNEP}, namely 
 \begin{align} \label{eq:symm_lin_eig}
\SSS \AAA \xxx = \lambda \SSS \BBB \xxx
 \end{align}
 is a symmetric eigenvalue problem. The vector $\eee$ has infinite length with ones in all the entries. The coefficients of the matrix $\CCC$ fulfill the following relations
\begin{subequations} \label{eq:c_mat}
\begin{align}
 c_{i,1}	&=\frac{1}{i+1}             && i \geq 1,					\label{eq:ci1}				\\
 c_{i-1,j}	&=\frac{j}{i} c_{i,j-1}		&& i,j > 1.		            \label{eq:cij}
\end{align}
\end{subequations}
\end{thm}

\begin{proof}
We start observing that $M_j$, for $j \geq 0$, are symmetric matrices as consequence of~\eqref{eq:Am_symm}. The relations~\eqref{eq:c_mat} uniquely define a matrix $\CCC$ since the first column is fixed in~\eqref{eq:ci1} and the $j$-th column is computed by the $(j-1)$-th column in~\eqref{eq:cij}. We start by showing that the matrix $\CCC$ is symmetric. Let us consider $i>j$, namely $i=j+k$ for some positive integer $k$. By iteratively using \eqref{eq:cij} we obtain the relations
\begin{align*}
 (j+1) c_{j+k,j} 		= &(j+k) c_{j+k-1,j+1}				\\
 (j+2) c_{j+k-1,j+1}	 	= &(j+k-1) c_{j+k-2,j+2}			\\
 \cdots				= &	\cdots					\\
 (j+s) c_{j+k-s+1,j+s-1} 	= &(j+k-s+1) c_{j+k-s,j+s}			\\	
  \cdots				= &	\cdots				\\
 (j+k) c_{j+1,j+k-1} 		= &(j+1) c_{j,j+k},			
\end{align*}
that combined together give 
\begin{align*}
 c_{j+k,j}  = \frac{ (j+k) (j+k-1) \dots (j+1) }{(j+1) (j+2) \dots (j+k) } c_{j,j+k},
\end{align*}
that is, $c_{i,j}=c_{j,i}$. The case $i<j$ is analogous and we conclude that the matrix $\CCC$ is symmetric.

By multiplying~\eqref{eq:nep_lin} on the left by the matrix~\eqref{eq:symmetrizer} we get  
\renewcommand{\arraystretch}{1.3}
\begin{align}\label{eq:SA}
 \SSS \AAA= 
  \begin{pmatrix}
   -M_0	&			&			&		   	&       	&       \\
	&	c_{1,1} M_{2}	&	c_{1,2} M_{3}	&	c_{1,3} M_{4}	& c_{1,4} M_{5} & \dots	\\
	&	c_{2,1} M_{3}	&	c_{2,2} M_{4}	&	c_{2,3} M_{5}	&       	&	\\
	&	c_{3,1} M_{4}	&	c_{3,2} M_{5}	&		   	&       	&	\\
	&	c_{4,1} M_{5} 	&			&		   	&       	&       \\
	&   \vdots      	&               	&			&       	&
  \end{pmatrix}
 \end{align}
 \renewcommand{\arraystretch}{1}
and
\renewcommand{\arraystretch}{1.3}
\begin{align} \label{eq:SB}
\SSS \BBB = 
    \begin{pmatrix}
 \frac{1}{1}		M_{1}	&	\frac{1}{2}		M_{2}	& 	\frac{1}{3}		M_{3}	&	\frac{1}{4}	M_{4}		& \frac{1}{5}	M_{5} 	& \dots		\\
 \frac{c_{1,1}}{1}	M_{2}	&	\frac{c_{1,2}}{2} 	M_{3}	& 	\frac{c_{1,3}}{3} 	M_{4}	&	\frac{c_{1,4}}{4} M_{5}		& 			& 		\\
 \frac{c_{2,1}}{1}	M_{3}	&	\frac{c_{2,2}}{2}	M_{4}	& 	\frac{c_{2,3}}{3} 	M_{5}	&					& 			&		\\
 \frac{c_{3,1}}{1}	M_{4}	&	\frac{c_{3,2}}{2} 	M_{5}	& 					&					& 			&		\\
 \frac{c_{4,1}}{1} 	M_{5}	&					&					&					& 			&		\\
 \vdots				&					&					&					& 			&	
\end{pmatrix}.
\end{align} \renewcommand{\arraystretch}{1}
The matrix $\SSS \AAA$ is symmetric because $\CCC$ and $M_j$, for $j \geq 0$, are symmetric. By using~\eqref{eq:ci1} we get that the first block-row of $\SSS \BBB$ is equal to its first block column, whereas the equation~\eqref{eq:cij} and the symmetry of $\CCC$ gives the relation
\begin{align*}
 \frac{c_{i-1,j}}{j}=\frac{c_{j-1,i}}{i}=\frac{c_{i,j-1}}{i},
\end{align*}
which directly implies that the $(i,j)$-th and the $(j,i)$-th blocks of $\SSS \BBB$ are equal. Hence the matrix $\SSS \BBB$ is symmetric and~\eqref{eq:symm_lin_eig} is a symmetric eigenvalue problem.
\end{proof}

\begin{rmk}
The eigenvalue problems~\eqref{eq:linNEP} and~\eqref{eq:symm_lin_eig} have the same eigenpairs if the symmetrizer~\eqref{eq:symmetrizer} is nonsingular, namely $\SSS \xxx = 0$ only for $\xxx = 0$. In the next section we assume that $[\SSS]_{2N}$ is invertible for an $N$ large enough. This condition can be phrases in terms of solvability of a specific matrix equation as discussed in Observation~\ref{obs:S_singular}.
 %This hypothesis is verified, e.g., if the matrices $\sum_{m=1}^p A_m$ and $C_k \otimes F_k$ are nonsingular. However we have observed that the method works anyway...EXPLAIN BETTER, THIS RELATES TO WHAT COMES AFTER
 %This hypothesis is only needed for theoretical purposes and we did not observe 

\end{rmk}

The method that we refer to as \emph{infinite Lanczos} consists of applying the indefinite Lanczos method (Algorithm~\ref{alg:IndefLanczosReduced}), described in the next section, to the symmetric eigenvalue problem~\eqref{eq:symm_lin_eig}.

\section{Infinite Lanczos method} \label{sec:Lanc_finite}

\subsection{Indefinite Lanczos method} \label{sec:LanFin}
Eigenpair approximations to the generalized eigenvalue problem $Ax=\lambda Bx$, with $A,B \in \CC^{n \times n}$ symmetric matrices, not necessarily Hermitian, and $A$ nonsingular, can be obtained by using the indefinite Lanczos method~\cite[Section 8.6.1]{bai2000templates} that is summarized in Algorithm~\ref{alg:IndefLanczosReduced}. The method consists of computing an orthogonal basis of the Krylov space 
\begin{align}\label{eq:Krylov_spave_def}
 \mathcal{K}_k(A^{-1} B, q_1) := \sspan \left(
 q_1, A^{-1} B q_1, (A^{-1} B)^2 q_1, \dots, (A^{-1} B)^{k-1} q_1 
 \right)
\end{align}
by using, instead of the (standard) Euclidean scalar product, the \emph{indefinite scalar product} defined by the matrix $B$, namely $x^T B y$ is the $B$-product between $x,y \in \CC^n$.
The fact that $A^{-1} B$ is self-adjoint, with respect to this indefinite scalar product, leads to the property that the $B$-orthogonal basis of the Krylov space can be computed with a three-term recurrence. In particular, at the $k$-th iteration of Algorithm~\ref{alg:IndefLanczosReduced}, the following relations are fulfilled
\begin{subequations} \label{eq:ArnFactFull}
\begin{align} 
 &A^{-1} B Q_k = Q_{k+1} T_{k+1,k},	                          	\label{eq:ArnFact}    \\
 &Q_{k+1}^T B Q_{k+1} = \Omega_{k+1},				    	\label{eq:B-orth}
\end{align}
\end{subequations}
where the diagonal matrix $\Omega_{k+1}:=\diag(\omega_1, \dots, \omega_{k+1})$ and the tridiagonal matrix  $T_{k+1,k}=[t_{i,j}]_{i,j=1}^{k+1}$ contain the orthogonalization and normalization coefficients. The matrix $Q_{k+1}$ is $B$-orthogonal in the sense of~\eqref{eq:B-orth} and its columns, generated with a three-term recurrence, span the Krylov space~\eqref{eq:Krylov_spave_def}. The Ritz pairs of~\eqref{eq:ArnFact}, defined as follows, 
\begin{align} \label{eq:RitzPairsFinite}
 (\lambda,Q_k z), \ \ \mbox{ where } \ \  T_{k} z = \lambda  \Omega_k z,
\end{align}
provide an approximation to the eigenpairs of the original problem.
Since the indefinite scalar product defined by $B$ is in general degenerate, there may be cases of break down in Algorithm~\ref{alg:IndefLanczosReduced}. We refer to \cite[Section 8.6.1]{bai2000templates} and reference therein for a detailed discussion of this issue.

% REDOUCE THE NUMBER OF B-PRODUCTS
\begin{algorithm} 
\caption{Indefinite Lanczos} \label{alg:IndefLanczosReduced}
\SetKwInOut{Input}{input}\SetKwInOut{Output}{output}
\Input{Matrices $A,B \in \RR^{n \times n}$ and starting vector $q_1 \in \RR^n$}
\Output{Eigenpair approximations}
\BlankLine
\nl Set $q_{0}=0$, $t_{0,1}=0$, $\omega_1=q_1^T B q_1$
\BlankLine
\For{$k = 1,2,\dots$}{
\nl $w=A^{-1} B q_k$	\label{step:AinvB}											\\
\nl $z=B w$		\label{step:zcomp}													\\
\nl $[\alpha, \beta, \gamma]=z^T [q_k, q_{k-1}, w]$ \label{step:orth_coeff}                 \\
\nl $t_{k,k}=\alpha/\omega_k, t_{k-1,k}=\beta/\omega_{k-1}$										\\
\nl $w_{\perp}=w-t_{k,k} q_k - t_{k-1,k} q_{k-1}$											\\
\nl $t_{k+1,k}=\| w_{\perp} \|$	
\\
\nl $q_{k+1}=w_{\perp}/t_{k+1,k}$
\\
\nl $\omega_{k+1}=(\gamma-2 t_{k,k} \alpha - 2 t_{k-1,k} \beta + t_{k,k}^2 \omega_k + t_{k-1,k}^2 \omega_{k-1} )/t_{k+1,k}^2$		\\
}
\nl Extract eigenpair approximations.
\end{algorithm}

\subsection{Infinite Lanczos method in finite arithmetic}

We now derive a method that consists of applying the indefinite Lanczos method (Algorithm~\ref{alg:IndefLanczosReduced}) to the symmetric eigenvalue problem  
\begin{align} \label{eq:sym_truncated}
 [\SSS \AAA]_{N} x = \lambda [\SSS \BBB]_{N} x
\end{align}
obtained by extracting the main block sub-matrices from~\eqref{eq:symm_lin_eig}, where $N$ is a non-fixed parameter greater than the number of iterations performed in Algorithm~\ref{alg:IndefLanczosReduced}. The method we derive is independent on $N$ and, under the assumption that $\SSS$ given in~\eqref{eq:symmetrizer} is invertibile, corresponds to apply Algorithm~\ref{alg:IndefLanczosReduced} directly to the linear infinite dimensional eigenvalue problem~\eqref{eq:symm_lin_eig} with a specific starting vector. This equivalence is formally presented in Theorem~\ref{thm:inf_dim_equiv} at the end of this section.

%We only assume that starting vector $\bm q_1$ has  a finite number of nonzero elements, without loss of generality, only the first block is nonzero. 
%We now show how each individual step of Algorithm~\ref{alg:IndefLanczosReduced} can be carried out in finite arithmetic. We start with Step~\ref{step:AinvB} that can be performed as stated in the following result.
Algorithm~\ref{alg:IndefLanczosReduced} can be efficiently applied to~\eqref{eq:sym_truncated} by exploiting the structure of the matrices~\eqref{eq:sym_truncated}. We start with Step~\ref{step:AinvB} that can be performed as stated in the following result.

\begin{thm}[Action of \mbox{$ [\SSS \AAA]_N^{-1} [\SSS \BBB]_N $} ] \label{thm:action_op}
 Assume that $[\SSS]_{2N}$ is invertible, let $q_k \in \CC^{Nn}$ be such that only its first $k$ blocks are nonzero, corresponding to the columns of $Q_k:=[\tilde q_1, \dots, \tilde q_k]$. Then, only the first $k+1$ blocks of  $w=[\SSS \AAA]_N^{-1} [\SSS \BBB]_N q_k$ are nonzero, corresponding to the columns of $W:=[w_1, \dots, w_{k+1}]$ given by
\begin{align} \label{eq:W}
 W=w_1 e_1^T + Q_k D,
\end{align}
where $D \in \RR^{k \times (k+1)}$ is a diagonal matrix with coefficients $d_{j,j+1}=1/j$ and
\begin{align} \label{eq:w1}
 w_1=M_0^{-1} \sum_{j=1}^k \frac{M_{j}}{j} \tilde q_j.
 %w_{j}=\frac{q_{j-1}}{j-1} && j=2, \dots, k+1
\end{align}
\end{thm}
\begin{proof}
By using the specific structure of the matrices~\eqref{eq:SA}~and~\eqref{eq:SB}, the nonzero blocks of $w$ fulfill $\vecc(W) = ([\SSS \AAA]_N)^{-1} ([\SSS \BBB]_N) \vecc([Q_k, \ 0])$. We can then derive the following relations
\begin{align*}
 \begin{pmatrix} \vecc(W) \\ 0 \end{pmatrix}
 &= ([\SSS]_{2N} [\AAA]_{2N})^{-1} ([\SSS]_{2 N} [\BBB]_{2 N}) \begin{pmatrix} \vecc([Q_k, \ 0]) \\ 0 \end{pmatrix} \\
 &= ([\AAA]_{2N})^{-1} ([\BBB]_{2N}) \begin{pmatrix} \vecc([Q_k, \ 0]) \\ 0 \end{pmatrix}.
\end{align*}
Hence $\vecc(W) = ([\AAA]_{N})^{-1} ([\BBB]_{N}) \vecc([Q_k, \ 0])$, this directly implies \eqref{eq:W}~and~\eqref{eq:w1}, c.f., \cite[Section 4.2]{Jarlebring_INFARN_2012}. 
\end{proof}

By using the previous result, we conclude that in Algorithm~\ref{alg:IndefLanczosReduced}, if $q_1$ has only the first block which is nonzero, then $q_k$ at the $k$-th iteration will have $k$ nonzero blocks. This is due to the fact that, none of the steps, except Step~\ref{step:AinvB}, introduce fill-in in the vectors $q_1, q_2, \dots, q_k$. In Step~\ref{step:orth_coeff} of Algorithm~\ref{alg:IndefLanczosReduced} the products $z^T q_{k}$, $z^T q_{k-1}$ and $z^T w$ are computed. Observe that the vectors multiplied by $z$ have at most $k+1$ nonzero blocks. Therefore, even if $z=[\SSS \BBB]_N w$ is in general a full vector, only the first $k+1$ blocks are required. These blocks can be computed as follows.

\begin{thm}[Action of \mbox{$[\SSS \BBB]_N$}] \label{thm:SB_action}
Let us consider $z:=[\SSS \BBB]_N w$, where $w$ is given as in Theorem~\ref{thm:action_op}. Then the first $k+1$ blocks of $z$, corresponding to the columns of $Z:=[z_1, \dots, z_{k+1}]$, fulfill the relation
\begin{align} \label{eq:Z_naive}
 Z = \sum_{m=1}^p A_m W (G_{k+1} \circ F_m),  
\end{align}
where $G_{k+1} \in \RR^{(k+1) \times (k+1)}$ has coefficients $g_{j,1}=g_{1,j}=1/j$ for $j=1, \dots, k+1$ and $g_{i,j}=c_{i-1,j}/j$ and 
\begin{align*}
 F_m := 
    \begin{pmatrix}
 f_m^{(1)} (0)	    & f_m^{(2)}(0)	  & f_m^{(3)}(0)    & \dots    & f_m^{(k+1)}(0)	\\
 f_m^{(2)} (0)	    & f_m^{(3)}(0)	  & \dots		    & 	 	   & f_m^{(k+2)}(0)	\\
 f_m^{(3)} (0)	    &  \dots		  &		            &		   & f_m^{(k+3)}(0)	\\
  \vdots	        &     			  &			        &		   & \vdots		\\
 f_m^{(k+1)} (0)	& f_m^{(k+2)}(0)  & f_m^{(k+3)}(0)	& \dots	   & f_m^{(2k+1)}(0)
 \end{pmatrix}.
 \end{align*}
\end{thm}

\begin{proof}
 Since $w$ has only the first $k+1$ blocks that are nonzero, we can express 
\begin{align} \label{eq:zj}
 \vecc Z
  = [\SSS \BBB]_{k+1} 
 \vecc W.
 \end{align} 
 By using~\eqref{eq:SB} and that $M_j = \displaystyle \sum_{m=1}^p f_m^{(j)}(0) A_m$, we can decompose
 \begin{align} \label{eq:SB_k+1}
 [ \SSS \BBB]_{k+1} &= \sum_{m=1}^p (G_{k+1} \circ F_m) \otimes A_m.
 \end{align}
 Equation~\eqref{eq:Z_naive} follows by combining~\eqref{eq:SB_k+1}~and~\eqref{eq:zj}, and by using the properties of the Kronecker product.
\end{proof}

\begin{obs} \label{obs:matrix_scalar_product}
 The scalar product between the vectorization of two matrices can be carried out directly in matrix form by using the Hadamard product as follows: $(\vecc Z)^T \vecc W = \tilde e^T (Z \circ W) e$ with $\tilde e \in \RR^n$ and $e \in \RR^k$.
\end{obs}

\begin{obs} \label{obs:S_singular}
With the same reasoning as in Theorem~\ref{thm:SB_action}, we can decompose \eqref{eq:symmetrizer} as   
\begin{align*}
 [\SSS]_{2N} &= \sum_{m=1}^p \left[ 
 \begin{pmatrix}
  1 &               \\
    & [\CCC]_{2N-1}
 \end{pmatrix}
 \circ F_m \right] \otimes A_m.
 \end{align*}
Therefore, we can relate the invertibility of $[\SSS]_{2N}$ with the solvability of the following linear matrix equation 
\begin{align*}
 \sum_{m=1}^p A_m X \left[
  \begin{pmatrix}
  1 &               \\
    & [\CCC]_{2N-1}
 \end{pmatrix}
 \circ F_m \right] = B 
\end{align*}
for any $B \in \CC^{2N \times n}$. Linear matrix equations are extensively studied in recent literature. See the review paper~\cite{simoncini2016computational} and reference therein. In the numerical examples reported in Section~\ref{sec:num_exp} we never encounter a case when $[\SSS]_{2N}$ was singular. A case when this matrix is obviously singular is when the NEP is defined by polynomial functions. Although the theory does not cover this case, we have successfully applied the method we are deriving without introducing any breakdown or instability.
\end{obs}

Figure~\ref{Fig:illustration_alg} illustrates the structure of the matrices and vectors, involved in Algorithm~\ref{alg:IndefLanczosReduced}, when applied to~\eqref{eq:sym_truncated} with a starting vector that has only the first block which is nonzero. At iteration $k$ only the vectors $q_{k-1}, q_{k}$ are needed, therefore they are the only vectors that need to be stored. 
\begin{figure}[t]
\centering
\includegraphics{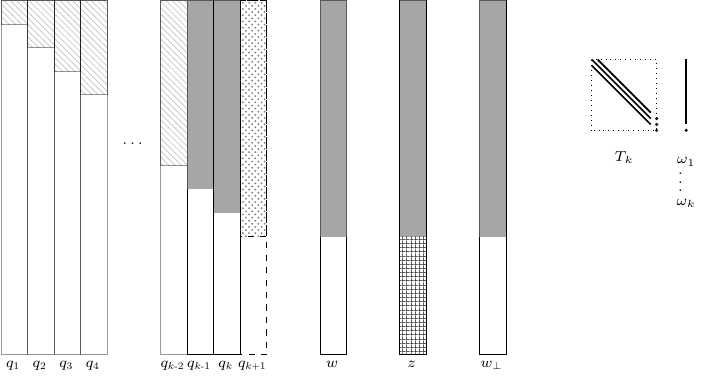}
\caption{Illustration of the structure of the matrices and vectors involved in Algorithm~\ref{alg:IndefLanczosReduced}, at iteration $k$, when applied to~\eqref{eq:sym_truncated}. The vectors $q_1, \dots, q_{k-2}$ in transparency are produced in the previous iterations and are not used after the $(k-1)$-iteration. The oblique lines pattern indicates the nonzero blocks of these vectors.
The lower part of the vector $z$ with grid pattern is not computed. The vector $q_{k+1}$ is computed at the end of the iteration and the dotted pattern indicates the nonzero blocks.}
\label{Fig:illustration_alg}
\end{figure}
\begin{algorithm} 
\caption{Infinite Lanczos} \label{alg:InfIndefLanczosReduced}
\SetKwInOut{Input}{input}\SetKwInOut{Output}{output}
\Input{NEP \eqref{eq:nep} and starting vector $Q_1 \in \RR^{n \times 1}$}
%\Output{An indefinite Arnoldi factorization $(\bm Q_{m+1},T_{m+1,m})$}
\Output{Eigenpair approximations}
\BlankLine
\nl Set $Q_{0}=0$, $t_{0,1}=0$, $\omega_1=Q_1^T M_1 Q_1$ 
\BlankLine
\For{$k = 1,2,\dots$}{
\nl Compute $w_1$ and $W=w_1 e_1^T + Q_k D$ as in \eqref{eq:W} and \eqref{eq:w1}	\label{step:W}				              \\
\nl Compute $Z$, defined in~\eqref{eq:Z_naive}, as in Section~\ref{sec:Z_comp} \label{step:Zcomp}			\\
\nl (Extend with zeros $Q_k:=[Q_k \ 0]$ and $Q_{k-1}:=[Q_{k-1} \ 0 \ 0]$)
    \\ Compute $[\alpha, \beta, \gamma] = \vecc(Z)^T [ \vecc(Q_k), \vecc(Q_{k-1}), W ] $ as in Observation~\ref{obs:matrix_scalar_product}			\\
\nl $t_{k,k}=\alpha/\omega_k, t_{k-1,k}=\beta/\omega_{k-1}$	\\
\nl $W_{\perp}=W-t_{k,k} Q_k - t_{k-1,k} Q_{k-1}$		\\
\nl $t_{k+1,k}=\| W_{\perp} \|_{F}$	
\\
\nl $Q_{k+1}=W_{\perp}/t_{k+1,k}$
\\
\nl $\omega_{k+1}=(\gamma-2 t_{k,k} \alpha - 2 t_{k-1,k} \beta + t_{k,k}^2 \omega_k + t_{k-1,k}^2 \omega_{k-1} )/t_{k+1,k}^2$		\\
}
\nl extract eigenpair approximations as in Section~\ref{sec:eigenpair_extract}
\end{algorithm}
Algorithm~\ref{alg:InfIndefLanczosReduced} is the combination of the results presented in this section. More precisely, Algorithm~\ref{alg:InfIndefLanczosReduced} is the reformulation of Algorithm~\ref{alg:IndefLanczosReduced}, applied to \eqref{eq:sym_truncated},  where the nonzero blocks of the vectors $q_k, q_{k-1}, w$, and the needed blocks of $z$, are stored as columns of the matrices $Q_k, Q_{k-1}, W$ and $Z$. Moreover, the size of the linearization~\eqref{eq:sym_truncated} is implicitly expanded at each iteration. Observe that at iteration $k$ only the first $2k+1$ derivatives $f^{(j)}_m(0)$ for $m=1, \dots , p$ and $j=1, \dots ,2k+1$ are needed. 
We now conclude this section by showing the equivalence between Algorithm~\ref{alg:IndefLanczosReduced}, directly applied to the infinite dimensional problem~\eqref{eq:symm_lin_eig}, and Algorithm~\ref{alg:InfIndefLanczosReduced}.

\begin{thm}[Infinite dimensional equivalence] \label{thm:inf_dim_equiv}
Assume that the matrix $\SSS$, given in~\eqref{eq:symmetrizer}, is invertible and let $\qqq_1$ be an infinite length vector with only the first block $q_1$ nonzero. Then, Algorithm~\ref{alg:IndefLanczosReduced}, with stating vector $\qqq_1$, is applicable to~\eqref{eq:symm_lin_eig} and the matrices $Q_k$, that have as columns the first $k$ nonzero blocks of $\bm{q}_k$, $\underbar T_k$, and $\omega_{k}$ are equal to the homonymous matrices generated by Algorithm~\ref{alg:InfIndefLanczosReduced} with starting matrix $Q_1=[q_1]$.
\end{thm}

\begin{proof}
We denote by $\bm{q}_1, \bm{q}_2, \dots, \bm{q}_{k}$ the infinite-length vectors generated by Algorithm~\ref{alg:IndefLanczosReduced}
and by $Q_1, Q_2, \dots, Q_k$ the matrices generated by Algorithm~\ref{alg:InfIndefLanczosReduced}.  
The proof is based on induction over the iteration count $k$. The result is trivial for $k=1$. Suppose the results holds for some $k$. In Step~\ref{step:AinvB} of Algorithm~\ref{alg:IndefLanczosReduced}, by using that $\SSS$ is invertible, we have 
\begin{align*}
 \bm{w} = (\SSS \AAA)^{-1} (\SSS \BBB) \bm{q}_k =  \AAA^{-1} \BBB \bm{q}_k.
\end{align*}
By using the induction hypothesis, $\bm{q}_k$ has $k$ nonzero blocks, corresponding to the column of the matrix $Q_k$ generated at the $(k-1)$-th iteration of Algorithm~\ref{alg:InfIndefLanczosReduced}. Because of the structure of the matrices~\eqref{eq:SA}~and~\eqref{eq:SB}, we get that $\bm{w}$ has only the first $k$ blocks which are nonzero, corresponding to the columns of the matrix $W$ that fulfills~\eqref{eq:W}. Therefore this matrix corresponds to the matrix computed in Step~\ref{step:W} of Algorithm~\ref{alg:InfIndefLanczosReduced}. In the Step~\ref{step:zcomp} of Algorithm~\ref{alg:IndefLanczosReduced} we compute 
$\bm z = \SSS \BBB \bm{w}$.
This vector is in general full. However, in the Step~\ref{step:orth_coeff} of Algorithm~\ref{alg:IndefLanczosReduced} the products 
$\bm{z}^T \bm{q}_k$, $\bm{z}^T \bm{q}_{k-1}$ and $\bm z^T \bm{w}$ are computed. By induction hypothesis, $\bm{q}_{k}$, $\bm{q}_{k-1}$ have respectively $k$ and $k-1$ nonzero blocks, corresponding to the columns of the matrices $Q_{k}$ and $Q_{k-1}$ generated by Algorithm~\ref{alg:InfIndefLanczosReduced}. Therefore, only the first $k+1$ blocks of $\bm{z}$ are required and they can be computed with the same reasoning of Theorem~\ref{thm:SB_action}. More precisely, the first $k+1$ blocks of $\bm{z}$ are the columns of the matrix $Z$ that fulfills~\eqref{eq:Z_naive}. Therefore this matrix coincides with the matrix generated by Step~\ref{step:Zcomp} of Algorithm~\ref{alg:InfIndefLanczosReduced}. In order to conclude that $\bm{q}_{k+1}$ has only we first $k+1$ nonzero blocks, corresponding to the columns of the matrix $Q_{k+1}$ generated by Algorithm~\ref{alg:InfIndefLanczosReduced}, we only need to use the property $\| M \|_F = \| \vecc(M) \|_2$ for every matrix $M$.

\end{proof}

\subsection{Robust extraction of eigenpair approximations} \label{sec:eigenpair_extract}
%[WORK IN PROGRESS]
%use a lot nonlinear Rayleigh–Ritz procedure
We propose to enhance Algorithm~\ref{alg:InfIndefLanczosReduced} as follows. We consider the projected NEP 
\begin{align} \label{eq:nep_projected}
V^T M(\lambda) V z = 0,
\end{align}
where $V \in \CC^{n \times k}$ is an orthogonal matrix. Under the assumption that $V$ posses good approximation properties, eigenpair approximations to the NEP~\eqref{eq:nep} are given by $(\lambda, Vz)$. This can be seen as the Galerkin projection method that uses the range of $V$ as projection space. This technique for extracting eigenpair approximations, called \emph{nonlinear Rayleigh–Ritz procedure} or \emph{subspace acceleration}, is often used to improve properties of more basic algorithms, e.g., the nonlinear Arnoldi method~\cite{voss2004arnoldi}, Jacobi–Davidson methods~\cite{effenberger2013robust,betcke2004jacobi}, infinite Arnoldi~\cite{jarlebring2010arnoldi}, block preconditioned harmonic projection methods~\cite{xue2018block}, and many more. 

In our framework, there is a natural choice for the projection space. The matrix $V$ is chosen as the orthogonal matrix whose columns span the subspace of vectors obtained by extracting the first column from $Q_1, \dots, Q_k$ generated by Algorithm~\ref{alg:InfIndefLanczosReduced}. The reason this matrix contains good approximation properties is due to the following argument. In the Ritz pairs extraction described in Section~\ref{sec:LanFin}, the eigenvector approximations of~\eqref{eq:linNEP} (or equivalently of~\eqref{eq:symm_lin_eig}) are given by the first block row of the Ritz vectors~\eqref{eq:RitzPairsFinite}. Thus, the eigenvector approximations to the NEP~\eqref{eq:nep} are also obtained by the first block of these Ritz vectors and thus by the first block row of the Krylov basis, namely by the columns of the proposed matrix $V$. %The reason why this approach is expected to be more robust with respect to the standard Ritz pair extraction concerns the fact that we are not directly using the matrices $T_k, \Omega_k$ that may be effected by roundoff errors due to the loss of orthogonality, but we are recomputing an orthogonal basis .

The projected problem~\eqref{eq:nep_projected} has size $k$ equal to the number of iterations, which is typically, in the context of Krylov methods, a small number, i.e., $k \ll n$. Therefore, the projected problem~\eqref{eq:nep_projected} has small size and solving~\eqref{eq:nep_projected} is not the computationally dominating part of Algorithm~\ref{alg:InfIndefLanczosReduced}. In the numerical experiments we have tested, for solving~\eqref{eq:nep_projected}, the following methods: Beyn's contour integral method~\cite{beyn2012integral}, NLEIGS~\cite{guttel2014nleigs} and IAR~\cite{Jarlebring_INFARN_2012}. The choice the method for  solving~\eqref{eq:nep_projected} depends on the features of the original problem~\eqref{eq:nep} and there is not a favorite candidate. For example, one may want to exploit that the projected problem~\eqref{eq:nep_projected} is defined by the same nonlinear functions and may inherit several features of the original NEP~\ref{eq:nep} such as being symmetric or palindromic. Therefore, this problem can be solved in a more robust way with structure preserving methods.

\section{Indefinite scalar product computation} \label{sec:Z_comp}
Under the assumption that the linear systems with the matrix $M_0$ can be efficiently solved, e.g., exploiting the sparsity, the dominating part of Algorithm~\ref{alg:InfIndefLanczosReduced} is the Step~\ref{step:Zcomp}, namely the computation of~\eqref{eq:Z_naive}, which has complexity~$\OOO(k^2n)$. In this section we derive efficient methods for computing this quantity.

\subsection{Computation of Step~\ref{step:Zcomp}: General case}

The following theorem provides an effective approximation to~\eqref{eq:Z_naive} without any specific assumption on the coefficients of~\eqref{eq:spmf}.
\begin{thm}
 Let $U,V \in \RR^{n \times q}$ the factors of the best rank $q$ approximation, with respect to the  Euclidean norm, 
 to the matrix $G_{k+1}$. Then
 \begin{align} \label{eq:Z_general_approx}
 \tilde Z = \sum_{m=1}^p A_m \sum_{j=1}^q W \diag(u_j) F_m \diag(v_j)
\end{align}
 is such that 
  \begin{align} \label{eq:Z_approx_error}
   \| Z - \tilde Z \|_F \le 
  \left( \sum_{m=1}^p \| A_m W\|_F \|F_m \|_F   \right)  \sum_{j=q+1}^k \sigma_j(G_k).
  \end{align}
\end{thm}

\begin{proof}
The approximation \eqref{eq:Z_general_approx} is obtained by replacing $G_{k+1}$ with $U V^T$ in \eqref{eq:Z_naive} and using $u_j v_j^T \circ F_m = \diag(u_j) F_m \diag(v_j)$. The equation~\eqref{eq:Z_approx_error} follows by the triangular inequality and by the fact that the Frobenius norm is sub-multiplicative with respect to the Hadamard product.
\end{proof}

The approximation~\eqref{eq:Z_general_approx} is effective, with $q$ small, since the matrix $G_k$, which is problem independent, has a fast decay in the singular values. In Figure~\ref{fig:G_svd} the singular values\footnote{The singular values are computed in \texttt{BigFloat} arithmetic using the package \texttt{GenericSVD}.} of this matrix are displayed for different sizes $k$. Moreover, the computation of~\eqref{eq:Z_general_approx}, requires less computation time than~\eqref{eq:Z_naive} since the products with the Hankel matrices $F_j$ can be efficiently computed with FFTs~\cite[Section 4]{luk2000fast}. The complexity for computing~\eqref{eq:Z_general_approx} is $\OOO(n k \log k)$.

\begin{figure}
\centering
\includegraphics{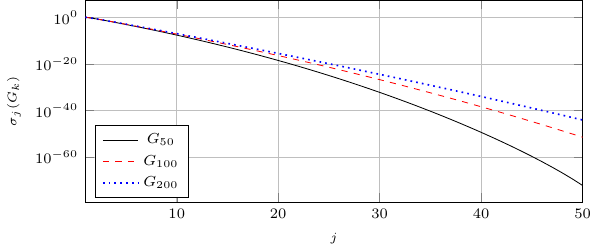}
\caption{Singular values decay of the matrices $G_k$ for $k=50,100,200$.} %Only the $50$ largest singular values are displayed.} 
\label{fig:G_svd}
\end{figure}

\subsection{Computation of Step~\ref{step:Zcomp}: delay eigenvalue problem}
The stability analysis of delay systems of the form
\begin{align} \label{eq:DDE_diff}
 \dot x(t) = A_2 x(t) + \sum_{m=3}^p A_m x(t- \tau_m)
\end{align}
is related to solving NEPs, referred to as \emph{delay eigenvalue problems}, see~\cite[Ch. 2]{Jarlebring:2008:THESIS} and~\cite[Ch. 1 Sect. 1.2]{michiels2007stability}, defined as follows: 
\begin{align}\label{eq:DEP}
 M(\lambda)  = - \lambda I + A_2 + \sum_{m=3}^p A_m e^{-\lambda \tau_m}.
\end{align}
In this case the matrices $F_m$ have at most rank one. More precisely, a direct computation leads to $F_1=-e_1 e_1^T$, $F_2=0$ and for $m \geq 3$ we get $F_m=-\tau_m v v^T$ with $v_j=(-\tau_m)^{j-1}$.
Therefore, the computation time of~\eqref{eq:Z_general_approx} is much lower than~\eqref{eq:Z_naive} since it involves only products with low rank matrices. The complexity for computing of~\eqref{eq:Z_general_approx} is reduced, by exploiting the low-rank structure, to~$\OOO(n p k)$.

\subsection{Computation of Step~\ref{step:Zcomp}: polynomial plus low-rank structured NEPs}
In certain applications \eqref{eq:spmf} can be written as sum of a polynomial and a low-rank part. See, e.g., \cite[Ch. 2 Sec. 4]{effenberger2013robust}, \cite[Ch. 1 Sec. 1.2]{van2015rational}, \cite[\texttt{gun} problem]{betcke2013nlevp} and \cite[Sec. 6.2.2]{betcke2017restarting}. More precisely:
\begin{align} \label{eq:spmf_poly_lr}
 M(\lambda) = \sum_{m=1}^d \lambda^{m-1} A_m  + \sum_{m=d+1}^p f_m(\lambda) U_m U_m^T 
\end{align}
with $U_m, U_m \in \CC^{n \times r_m} $ and $r_m \ll n$. In this case we split \eqref{eq:Z_naive} in the polynomial and low-rank terms, namely $Z=Z_p+Z_{lr}$ with 
\begin{subequations} \label{eq:Z_poly_low_rank}
\begin{align} 
  Z_p := \sum_{m=1}^d A_m \left[ W (G_{k+1} \circ F_m) \right], \label{eq:Zp} \\
  Z_{lr} :=  \sum_{m=d+1}^p \left[ U_m (U_m^T W) \right] (G_{k+1} \circ F_m). \label{eq:Zlr}
\end{align}
\end{subequations}
%In the first sum we take advantage of the fact that $F_t$ are ĺow rank matrices, whereas in the second sum 
The term~\eqref{eq:Zp} can be efficiently approximated as in~\eqref{eq:Z_general_approx} by exploiting the low-rank structure of the matrices $F_m$. Moreover, since $U_m U_m^T$ in~\eqref{eq:Zlr} have low rank, the computation of $Z$ with \eqref{eq:Z_poly_low_rank}, respecting the order given by the parentheses in~\eqref{eq:Zlr}, requires less computation time than~\eqref{eq:Z_naive}. The complexity of~\eqref{eq:Z_poly_low_rank} is $\OOO((d+r)n)$ where $r=\displaystyle \max_{d+1 \le t \le p} r_m$.

\section{Numerical simulations} \label{sec:num_exp}

In the following numerical experiments\footnote{All simulations were carried out with Intel octa core i7-4770 CPU 3.40GHz and 24 GB RAM.} we use, as error measure, the relative error defined as follows 
\begin{align*}
\Err(\lambda,x) := \frac{ \| M(\lambda) x \|_2 }{ \sum_{m=1}^p| f_m(\lambda) | \| A_m \|_{\infty} \| x \|_2 }.
\end{align*}
An eigenpair approximation is marked as ``converged'' if $\Err(\lambda,x)<10^{-8}$. The software used in these simulations is implemented in the Julia programming language~\cite{bezanson2017julia}, and publicly available in the Julia package NEP-PACK~\cite{NEP-PACK} \footnote{The scripts reproducing several of the presented examples are directly available in the web-page: 
 \url{https://people.kth.se/~gmele/InfLan/}}.

\subsection{Delay eigenvalue problem} \label{sec:num_dep}
We consider the delay eigenvalue problem arising from the spatial discretization of the following partial delay differential equation
\begin{align} \label{eq:PDDE_num}
 & u_{t}(\xi,t) = -\Delta u(\xi,t) + a(\xi) u(\xi,t-1) \\ 
 & \xi=(\xi_1,\xi_2) \in [0, \pi]^2, \ t>0   \nonumber
\end{align}
where $a(\xi)=-\xi_1 \sin(\xi_1+\xi_2)$, resulting in a problem of the form~\eqref{eq:DDE_diff} where all the matrices are real and symmetric. The spatial domain is partitioned with a uniform equispaced grid with $N$ points in each direction. The Laplace operator is discretized by the 5-points stencil finite difference approximation, leading to a NEP of size $n=N^2$, cf.~\cite[Section 6.2.1]{betcke2017restarting}. More precisely, the NEP is defined as $M(\lambda)=-\lambda I + A_2 +  e^{-\lambda} A_3$ where $I \in \RR^{n \times n}$ and the other matrix coefficients are given by 
{\small
\begin{align*}
& D := \frac{1}{h^2}
 \begin{pmatrix}
      -2    & 1         &               &       \\
      1     &  \ddots   &    \ddots     &       \\
            &  \ddots   &              &   1    \\
            &           &    1         & -2     \\
     \end{pmatrix} \in \RR^{N \times N},
 && 
 \begin{array}{l}
   F := \vecc \left( \left[ a(\xi_i,\xi_j) \right]_{i,j=1}^N \right) \in \RR^{N \times N},  \\ \\
 \tilde I \in \RR^N,
 \end{array} 
 \\
& A_2:=D \otimes \tilde I + \tilde I \otimes D \in \RR^{n \times n},      && A_3 := \diag( F ) \in \RR^{n \times n} 
\end{align*}
}%
with $h:=\pi/(N-1)$ discretization step. We run $50$ iterations of Algorithm~\ref{alg:InfIndefLanczosReduced}. In Figure~\ref{fig:dep_spectrum} and Figure~\ref{fig:dep_err_hist} we illustrate the robustness of the strategy for extracting the eigenpair approximations presented in Section~\ref{sec:eigenpair_extract}. More precisely, we compare the standar approach for extracting the eigenpair approximations, which is based on the computation of the Ritz pairs, with the more robust approach consisting of solving the projected NEP. In Figure~\ref{fig:dep_spectrum} is displayed the spectrum and the converged eigenvalues, respectively computed with the Ritz and the projected NEP approach. In this example we solve the projected NEP with the Beyn contour integral method\footnote{The disk of interest is set with center in the origin and radius $4$
with $N=1000$ discretization points and $\mbox{tol}_{\mbox{\tiny res}}=\mbox{tol}_{\mbox{\tiny rank}}=10^{-8}$.}~\cite{beyn2012integral}. The error history is presented in Figure~\ref{fig:dep_err_hist}. As expected, the convergence of the Algorithm~\ref{alg:InfIndefLanczosReduced}, with the Ritz pair approximation, appear to be slower with respect to the convergence of the Algorithm~\ref{alg:InfIndefLanczosReduced} with the more robust eigenpair extraction based on solving the projected NEP. 
\begin{figure}[t]
\begin{minipage}[t]{0.46\textwidth}
\hspace{-0.4cm}
\includegraphics{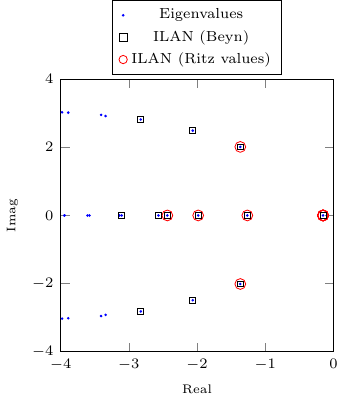}
\caption{Converged eigenvalues after $50$ iterations of Algorithm~\ref{alg:InfIndefLanczosReduced} applied to the NEP in Section~\ref{sec:num_dep}. The eigenvalues approximations are computed by extracting the Ritz pairs and by solving the projected NEP.}
\label{fig:dep_spectrum}
\end{minipage} \hspace{0.3cm}
\begin{minipage}[t]{0.46\textwidth}
\hspace{-0.4cm}
\includegraphics{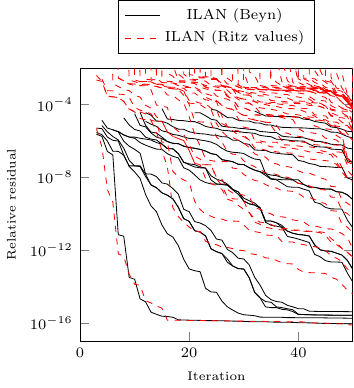}
\caption{Error history of Algorithm~\ref{alg:InfIndefLanczosReduced} for the NEP in Section~\ref{sec:num_dep}. In each iteration the eigenvalue approximations are computed by extracting the Ritz pairs and by solving the projected NEP. }
\label{fig:dep_err_hist}
\end{minipage}
\end{figure}

The performance of Algorithm~\ref{alg:InfIndefLanczosReduced} is affected by the method used for solving the projected NEP. In Table~\ref{tab:dep} we compare the time, and the number of computed eigenvalues, after $50$ iterations of Algorithm~\ref{alg:InfIndefLanczosReduced} combined with three different NEP solves for the projected problem: Beyn contour integral method (with the same settings as before), IAR~\cite{Jarlebring_INFARN_2012} ($50$ iterations)~\footnote{We perform $50$ iteration of IAR for solving the projected NEP and return the converged eigenvalues.} and IAR ($100$ iterations).

\setlength{\tabcolsep}{4.2pt}
\begin{table}[t!]
\begin{minipage}[t]{0.46\textwidth}
\begin{tabular}{cc|c|c|c|c|c|c|}  
\cline{3-8}
 &
 & \multicolumn{2}{|c|}{Beyn} 
 & \multicolumn{2}{|c|}{IAR: $50$ iter.} 
 & \multicolumn{2}{|c|}{IAR: $100$ iter.}   \\
 \hline
 \multicolumn{1}{|c}{prob. size} &
 & time 
 & conv. eig.
 & time 
 & conv. eig.
 & time 
 & conv. eig.                          \\
 \hline
 \multicolumn{1}{|c}{$10000$} &
 & 1.997 s
 & 11
 & 1.500 s
 & 9 
 & 2.570 s 
 & 13                            \\
  \hline
 \multicolumn{1}{|c}{$90000$} &
 & 14.191 s
 & 11
 & 13.120 s
 & 9 
 & 14.295 s
 & 19                            \\
  \hline
 \multicolumn{1}{|c}{$250000$} &
 & 36.177 s
 & 11
 & 35.496 s
 & 9 
 & 36.460 s  
 & 17                          \\
 \hline
\end{tabular}
\end{minipage}
\caption{Performance of Algorithm~\ref{alg:InfIndefLanczosReduced} applied to the NEP in Section~\ref{sec:num_dep} of different sizes. The NEPs are obtained by discretizing~\eqref{eq:PDDE_num} respectively with $N=100, 300, 500$ nodes in each direction. The projected problems are solved only at the last iteration with: Beyn contour integral method, IAR~\cite{Jarlebring_INFARN_2012} ($50$ iterations) and IAR ($100$ iterations).} 
\label{tab:dep}
\end{table}

\subsection{A benchmark problem} \label{sec:gun_num}
We now illustrate the performance of Algorithm~\ref{alg:InfIndefLanczosReduced} for solving a NEP that is symmetric but not Hermitian. We consider the \texttt{gun} problem that belong to the problem collection~\cite{betcke2013nlevp}. This NEP has the following form 
\begin{align} \label{eq:nep_gun}
 M(\lambda) = A_1 - \lambda A_2 + i \sqrt{\lambda} A_3  + i \sqrt{\lambda - \sigma_2^2} A_4  
\end{align}
where $\sigma_2 = 108.8774$. The matrices $A_j \in \RR^{9956 \times 9956}$, for $j=1, \dots, 4$, are real and symmetric. The NEP~\ref{eq:nep_gun} can be written in the form~\eqref{eq:spmf_poly_lr} since the matrix coefficients of the nonlinear part have low-rank, namely $\rk(A_3)=19$ and $\rk(A_4)=65$. The eigenvalues of interest are located inside the the closed disk centered in $250^2$ and with radius $5 \cdot 10^4$. Before applying Algorithm~\ref{alg:InfIndefLanczosReduced}, the problem is shifted and scaled. We set the parameters to $\lambda=\lambda_0 + \alpha \hat \lambda$ where $\lambda_0=300^2$ and $\alpha=(300-200)^2$. This problem has been solved with various 
methods~\cite{Jarlebring_INFARN_2012,guttel2014nleigs,van2015compact,lietaert2018compact,gaaf2017infinite} and, by numerical evidences, there are $21$ eigenvalues in the region of interest. For this problem we use IAR for solving the projected NEP. More precisely we test two variants: IAR (50 iterations) and IAR (200 iterations). As showed in the numerical experiment in Section~\ref{sec:num_dep}, the robustness of the whole Algorithm~\ref{alg:InfIndefLanczosReduced} is effected by the choice of the method used for solving the projected NEP. In Figure~\ref{fig:gun_spectrum} we can see that more eigenvalues converge when we solve more accurately the projected problem. The error history is presented in Figure~\ref{fig:gun_err_hist}. Solving more accurately the projected NEP is necessary to compute the outermost eigenvalues. 

\begin{figure}[t]
\begin{minipage}[t]{0.5\textwidth}
\hspace{-0.4cm}
\includegraphics{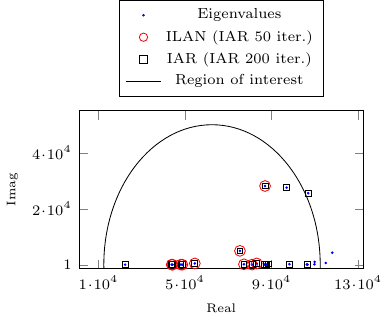}
\caption{Converged eigenvalues after $100$ iterations of Algorithm~\ref{alg:InfIndefLanczosReduced} applied to the NEP in Section~\ref{sec:gun_num}. The projected problem is solved with IAR (50 iterations) and IAR (200 iterations).}
\label{fig:gun_spectrum}
\end{minipage} \hspace{0.5cm}
\begin{minipage}[t]{0.45\textwidth}
\includegraphics{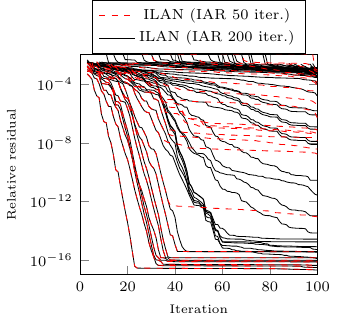}
\caption{Error history, for the NEP in Section~\ref{sec:gun_num}. In each iteration the projected problem is solved with IAR (50 iterations) and IAR (200 iterations).}
\label{fig:gun_err_hist}
\end{minipage}
\end{figure}

\subsection{A random symmetrized problem} \label{sec:symmetrized_nep}
In conclusion we illustrate how Algorithm~\ref{alg:InfIndefLanczosReduced} can be used for solving a nonsymmetric NEP. We introduce a symmetrization technique, consisting of doubling the problem size, based in the idea presented in~\cite[Sect. 5]{nour1989applications}. Namely, we define the symmetrized NEP as 
\begin{align}\label{eq:symmetrized_nep}
 \tilde M (\lambda) := 
 \begin{pmatrix}
  0                 &   M(\lambda)  \\
  M(\lambda)^T      &   0   
 \end{pmatrix} = 
 \sum_{m=1}^p f_m(\lambda) 
  \begin{pmatrix}
   0        &   A_m \\
   A_m^T    &   0
  \end{pmatrix}.
\end{align}
Observe that if $\left( \lambda, [y^T,x^T]^T \right)$ is an eigenpair of~\eqref{eq:symmetrized_nep}, then $(\lambda,x)$ is an eigenpair of $M(\lambda)$. We now consider the symmetrization, in the sense of~\eqref{eq:symmetrized_nep}, of the following NEP that is artificially constructed: 
\begin{align} \label{eq:unsymm_nep}
 M(\lambda) = A_1 - \lambda A_2 + \sin(\lambda) A_3 + e^{-\lambda} A_4
\end{align}
where $A_j \in \CC^{500 \times 500}$ are defined as follows: $A_1$ is the bidiagonal matrix with elements equal to $500$ in the upper and lower diagonal, $A_2$ is the identity matrix, $A_3=A_1/500$ and $A_4$ is a diagonal matrix with elements equal to $i$ (complex unit) in the lower diagonal. We perform $50$ iterations of Algorithm~\ref{alg:InfIndefLanczosReduced} and solve the projected NEP with NLEIGS~\footnote{We used the static variant with Leja-Bagby points automatically generated from the polygonal target. The shift is in zero and it is kept constant during the iterations.}~\cite{guttel2014nleigs} by targeting the eigenvalues contained in the rectangle with opposite vertices in $-1.5-1.5 i$ and $0.5+0.5 i$. In Figure~\ref{fig:symmetrized_nep_spectrum} is illustrated the spectrum and the converged eigenvalues, in the region of interest, after $50$ iterations of Algorithm~\ref{alg:InfIndefLanczosReduced}. The error history is illustrated in Figure~\ref{fig:symmetrized_nep_err_hist}.

\begin{figure}[t]
\begin{minipage}[t]{0.46\textwidth}
\hspace{-0.4cm}
\includegraphics{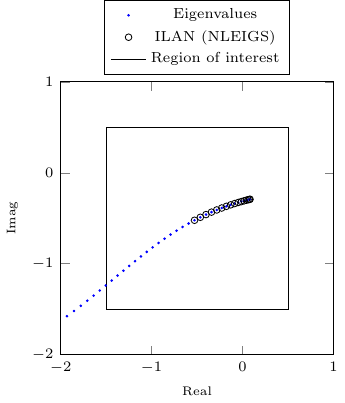}
\caption{Converged eigenvalues after $50$ iterations of Algorithm~\ref{alg:InfIndefLanczosReduced} applied to the NEP~\eqref{eq:symmetrized_nep} that is the symmetrization of~\eqref{eq:unsymm_nep}. The convergence is tested with respect to the original NEP~\eqref{eq:unsymm_nep}. The projected problem is solved with NLEIGS.}
\label{fig:symmetrized_nep_spectrum}
\end{minipage} \hspace{0.3cm}
\begin{minipage}[t]{0.46\textwidth}
\hspace{-0.4cm}
\includegraphics{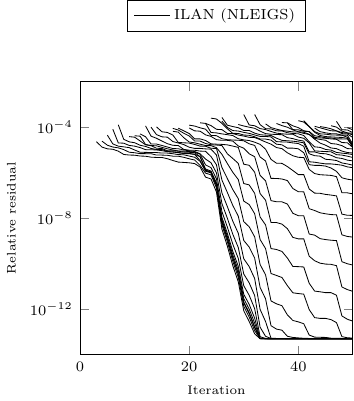}
\caption{Error history of Algorithm~\ref{alg:InfIndefLanczosReduced} applied to the  NEP~\eqref{eq:symmetrized_nep} that is the symmetrization of~\eqref{eq:unsymm_nep}. The error is computed with respect to the original NEP~\eqref{eq:unsymm_nep}. }
\label{fig:symmetrized_nep_err_hist}
\end{minipage}
\end{figure}

\section{Conclusions and outlook}
We have presented a method for solving symmetric NEPs. We have also illustrated how the problem structure, in particular the structure of the matrices and functions in~\eqref{eq:spmf}, can be exploited in order to reduce the computation time. However, there are NEPs that cannot be written in the format~\eqref{eq:spmf}, e.g., the waveguide eigevalue problem~\cite{jarlebring2017waveguide}, the reformulation of the Dirichlet eigenvalue problem with the boundary element method \cite{steinbach2009boundary,effenberger2012chebyshev}, etc. For some of these problems, only a routine for computing $M_k x$ is available. We believe that further research can potentially extend the applicability of Infinite Lanczos to such problems. 

In the numerical experiment in Section~\ref{sec:symmetrized_nep}, we have successfully solved a nonsymmetric NEP in the following way. Firstly we constructed a symmetric NEP~\eqref{eq:symmetrized_nep} whose eigenvalues are also eigenvalues of the original NEP. Then we applied the infinite Lanczos to the symmetrized problem~\eqref{eq:symmetrized_nep}. The matrices~\eqref{eq:symmetrized_nep} have clearly a very well defined block structure. We believe that infinite Lanczos can be further specialized for solving nonsymmetric NEPs by exploiting these structures. In conclusion we also believe that similar ideas can be extended to NEPs that are Hermitian, namely, $M(\lambda)^H=M(\bar \lambda)$ where $\bar \lambda$ represents the complex conjugate of $\lambda \in \CC$ and $M(\lambda)^H$ the Hermitian, or conjugate transpose, of the matrix $M(\lambda)$.

\section{Acknowledgement}
The author wishes to thank Sarah Gaaf (Utrecht University) who suggested a first approach to extend Theorem~\ref{thm:PEPsym} to the scaled companion linearization. The author also thanks Elias Jarlebring (KTH Royal Institute of Technology) who has provided feedback and help in the first stage of the project

\bibliographystyle{elsart-num-sort}      
\bibliography{manuscript}
\end{document}